\documentclass[10pt,a4paper]{amsart}
\usepackage[utf8]{inputenc}
\usepackage{amsmath, amscd, xcolor}
\usepackage{amsfonts}
\usepackage{amssymb}
\usepackage{amsthm}
\usepackage[all]{xy}

\newtheorem{thm}{Theorem}[section]
\newtheorem{prop}[thm]{Proposition}
\newtheorem{cor}[thm]{Corollary}
\newtheorem{lem}[thm]{Lemma}
\numberwithin{equation}{section}

\theoremstyle{definition}

\newtheorem{dfn}[thm]{Definition}

\theoremstyle{remark}
\newtheorem{rem}[thm]{Remark}

\renewcommand{\H}{\mathcal{H}}
\newcommand{\M}{\mathcal{M}}

\newcommand{\C}{\mathbb{C}}
\newcommand{\Z}{\mathbb{Z}}
\newcommand{\R}{\mathbb{R}}

\newcommand{\RR}{\mathcal{R}}
\newcommand{\Aut}[1]{{\rm Aut}(#1)}
\newcommand{\Out}[1]{{\rm Out}(#1)}
\newcommand{\Iso}[2]{{\rm Iso}(#1,#2)}

\begin{document}
\title{Locally Trivial W$^*$-Bundles}


\author{Samuel Evington}
\thanks{The first-named author is supported by an EPSRC Doctoral Training Grant (grant reference numbers EP/K503058/1 and EP/J500434). The second-named author was supported by the SFB~878 while he was employed by the University of M\"unster}
\address{SE: School of Mathematics and Statistics, University of Glasgow, 15 University Gardens, Glasgow G12 8QW, United Kingdom}
\email{s.evington.1@research.gla.ac.uk}	

\author{Ulrich Pennig}
\address{UP: Mathematics Institute, Cardiff University, Senghennydd Road, Cardiff CF24 4AG, United Kingdom}
\email{PennigU@cardiff.ac.uk}	

\begin{abstract}
We prove that a tracially continuous W$^*$-bundle $\M$ over a compact Hausdorff space $X$ with all fibres isomorphic to the hyperfinite II$_1$ factor $\RR$ that is locally trivial already has to be globally trivial. The proof uses the contractibility of the automorphism group $\Aut{\RR}$ shown by Popa and Takesaki. There is no restriction on the covering dimension of $X$.
\end{abstract}
\maketitle

\section{Introduction}
Tracially continuous W$^*$-bundles were introduced by Ozawa in \cite[Section 5]{Oz13}. They are similar in spirit to other notions of bundle in functional analysis, such as continuous $C(X)$-algebras \cite{Di63, Kas88} and Hilbert-$C(X)$-modules \cite{Di63, Ta79}. However, the fibres of a W$^*$-bundle are tracial von Neumann algebras and the topology is a mixture of the topology on the base space and the $2$-norm topology in the fibres. For example, the trivial W$^*$-bundle over a compact Hausdorff space $X$ with fibre the tracial von Neumann algebra $M$ is given by $C_{\sigma}(X,M)$, i.e.~the norm bounded, $2$-norm continuous maps $X \to M$. 

It was shown in \cite[Corollary 16]{Oz13} that a strictly separable W$^*$-bundle with all fibres isomorphic to the hyperfinite II$_1$ factor $\RR$ over a base space $X$ \emph{with finite covering dimension} is isomorphic to $C_{\sigma}(X,\RR)$. This automatic triviality is reminiscent of similar statements in the the context of Hilbert $C(X)$-modules \cite{Di63} and continuous $C(X)$-algebras with strongly self-absorbing fibres \cite{Hi07, DW08, DP13}. In the discussion that follows the proof, Ozawa raises the possibility of trivialisation results when $X$ is infinite dimensional (see also \cite[Question~3.14]{Bo15}). This leads one to ask what a non-trivial W$^*$-bundle over an infinite dimensional space $X$ with fibres isomorphic to $\RR$ could look like. 

We show, in our main result (Theorem~\ref{thm:hyperfinite_trivial}), that such a bundle would already have to be non-trivial \emph{locally}. More precisely, we consider W$^*$-bundles $\M$ that are locally trivial in the sense that every point $x \in X$ has a closed neighbourhood $Y \subseteq X$ such that the restriction $\M_Y$ is isomorphic to $C_{\sigma}(Y,\RR)$ as a W$^*$-bundle. We prove that this implies $\M \cong C_{\sigma}(X,\RR)$ as W$^*$-bundles. 

Local triviality no longer implies triviality for W$^*$-bundles with non-hyperfinite fibres. Indeed, we show in Section~\ref{sec:non-trivial} that there are non-trivial, but still locally trivial W$^*$-bundles already over such simple spaces as $S^1$. The fibres are given by II$_1$ factors with prescribed outer automorphism group, which were constructed in \cite{IPP08, VF08}. These II$_1$ factors do not absorb $\RR$ tensorially.

The study of W$^*$-bundles is motivated by work on the structure and classification of simple, nuclear C$^*$-algebras. In the light of the recent developments \cite{Ti15, El15, Go15}, the classification, by means of K-theoretic invariants, of simple, separable, nuclear, unital, infinite-dimensional C$^*$-algebras of finite nuclear dimension that satisfy the UCT is now complete. To identify finite nuclear dimension is, therefore, now a priority.

Given a simple, separable, nuclear, unital, infinite-dimensional C$^*$-algebra $A$ whose trace simplex $T(A)$ is a Bauer simplex, Ozawa showed that a certain tracial completion $\overline{A}^u$ of $A$ is a W$^*$-bundle over the space of extreme traces $\partial_eT(A)$ with fibres all isomorphic to $\mathcal{R}$. When $A$ has finite nuclear dimension, this bundle is trivial by combining results of \cite{Wi12} and \cite{Oz13}. In the reverse direction, the results of \cite{Oz13,MS12,KR14} (see also \cite{TWW15, Sa12}) and \cite{Bo15} (which builds on \cite{MS14, SWW15}) show that triviality of the bundle $\overline{A}^u$ combines with strict comparison, a mild condition on positive elements analogous to the order on projections in a II$_1$ factor being determined by their trace, to give finite nuclear dimension. This equivalence of regularity properties for C$^*$-algebras forms part of the Toms-Winter conjecture; see \cite[Section 6]{Ti15} for a full discussion.

The proof of our main result is based on the observation that each W$^*$-bundle $\M$ gives rise to a bundle $(B, p)$ in the sense of \cite[Chapter 2, Section 13.1]{FD88} where each fibre has the additional structure of a tracial von Neumann algebra (see Definition~\ref{dfn:TopologicalBundle}). The W$^*$-bundle $\M$ can be recovered by considering bounded, continuous sections of $(B,p)$. If $\M$ is locally trivial with all fibres isomorphic to $\RR$ in the above sense, then $(B,p)$ is locally trivial in the sense of algebraic topology and therefore associated to a principal bundle $P_B$ with structure group $\Aut{\RR}$ equipped with the $u$-topology. It follows from the contractibility of $\Aut{\RR}$ \cite{Po93} that $P_B$ has to be trivialisable, which translates into the triviality of $\M$.

The paper is organized as follows: Section $2$ contains the definitions of W$^*$-bundles and morphisms between them. We also recall the definition of fibres and generalise it to show that W$^*$-bundles can be restricted to closed subsets of the base space giving restriction morphisms $\M \to \M_Y$. This allows us to define local triviality. In Section $3$, we introduce the topological bundle $(B,p)$ of tracial von Neumann algebras associated to a W$^*$-bundle $\M$ over $X$. The topology on $B$ is such that one can retrieve $\M$ as the C$^*$-algebra of its bounded, continuous sections. This is analogous to the total space in the theory of $C(X)$-algebras (see for example \cite{Di63}). In Section $4$, we introduce the principal bundle $P_B$ and prove our main theorem. Finally, Section $5$ concerns the construction of non-trivial, locally trivial bundles with non-hyperfinite fibres.

\vspace{2mm}
\paragraph{\textbf{Acknowledgements}}
The research was undertaken whilst the first-named author was visiting the University of Münster, supported by the SFB 878 and a Graduate School Mobility Scholarship from the College of Science and Engineering at the University of Glasgow. He is grateful for the support and hospitality. Some parts of this work were completed while the second-named author was visiting the Center for Symmetry and Deformation (SYM) in Copenhagen. He would like to thank S{\o}ren Eilers for inviting him and the center administration for the hospitality. The authors would like to thank Stuart White for many helpful discussions and the referee for the comments on the original manuscript.





\section{W$^*$-bundles, Fibres and Restrictions}
This section contains the basic properties of tracially continuous W$^*$-bundles that we shall need in the sequel. We begin by stating the definition of a tracially continuous W$^*$-bundle to be used in this paper.
\begin{dfn} \label{dfn:WBundle}
A tracially continuous W$^*$-bundle over a compact Hausdorff space $X$ is a unital C$^*$-algebra $\M$ together with a unital embedding of $C(X)$ into the centre $Z(\M)$ of $\M$ and a conditional expectation $E:\M \rightarrow C(X)$ such that the the following axioms are satisfied:
\begin{itemize}
	\item[(T)] $E(a_1a_2) = E(a_2a_1)$, for all $a_1,a_2 \in \M$.
	\item[(F)] $E(a^*a) = 0 \Rightarrow a = 0$, for all $a \in \M$. 
	\item[(C)] The unit ball $\lbrace a \in \M: \|a\| \leq 1 \rbrace$ is complete with respect to the norm defined by $\|a\|_{2,u} = \|E(a^*a)^{1/2}\|_{C(X)}$.
\end{itemize}
\end{dfn}
\begin{rem}
This definition is a slight modification of Ozawa's original definition, which appears in \cite[Section 5]{Oz13}. The difference is that we do not require the base space to be metrisable. Close examination of \cite[Section 5]{Oz13} reveals that the metrisability of $X$ is not necessary for much of the basic theory of W$^*$-bundles. It is worth noting, however, that Owaza's condition of \emph{strictly separability} for a W$^*$-bundle (see \cite[Theorems 13 and 15]{Oz13}), which is equivalent to separability of $\M$ with respect to the $\|\cdot\|_{2,u}$-norm, implies that $C(X)$ is separable, so $X$ is metrisable. 
\end{rem}

We shall abbreviate \emph{tracially continuous W$^*$-bundle} to \emph{W$^*$-bundle} or even \emph{bundle} when there is no chance of confusion. We shall call $X$ the base space of the bundle, $\M$ the section algebra of the bundle, and $E$ the conditional expectation. We shall speak of the tracial axiom, the faithfulness axiom and the completeness axiom respectively. We shall often confuse the bundle itself with its section algebra, and speak of the bundle $\M$. 

It is instructive to consider the case where $X$ is the one point space $\lbrace * \rbrace$. We identify $\mathbb{C}$, $C(\lbrace * \rbrace)$, and $\mathbb{C}1_{\M}$. Now the data for a W$^*$-bundle over $X$ reduces to a unital C$^*$-algebra $\M$ and a state. The first two axioms then require this state to be a faithful trace. The effect of the third axiom is to ensure that the image of the unit ball of $\M$ under the GNS representation corresponding to this faithful trace is closed in the weak operator topology (see for example \cite[Lemma A.3.3]{Si08}). A W$^*$-bundle over a one point space is, therefore, just a \emph{tracial von Neumann algebra}: a (necessarily finite) von Neumann algebra together with a faithful, normal trace. 

We now recall the definition of morphisms between W$^*$-bundles and of trivial W$^*$-bundles, which are implicit in \cite[Section 5]{Oz13}.
\begin{dfn}
	Let $\M_i$ be a W$^*$-bundle over $X_i$ with conditional expectation $E_i$ for $i=1,2$. A morphism is a unital $*$-homomorphism $\alpha:\M_1 \rightarrow \M_2$ such that $\alpha(C(X_1)) \subseteq C(X_2)$ and the diagram
\begin{equation} \label{MorphismDef}
\xymatrix
{
	\M_1 \ar[r]^{\alpha} \ar[d]_{E_1} & \M_2 \ar[d]_{E_2} \\
	C(X_1) \ar[r]^{\alpha} & C(X_2) \\
}	
\end{equation}	
commutes. 
\end{dfn}
\begin{dfn}
For a given compact Hausdorff space $X$ and tracial von Neumann algebra $(M, \tau)$, the trivial W$^*$-bundle over $X$ with fibre $M$ is the C$^*$-algebra of $\|\cdot\|$-bounded, $\|\cdot\|_{2,\tau}$-continuous functions $X \rightarrow M$, denoted $C_\sigma(X,M)$, together with the following embedding and conditional expectation:
\begin{itemize}
	\item The embedding $\iota:C(X) \rightarrow Z(C_\sigma(X,M))$ is defined by $\iota(f)(x) = f(x)1_M$ for $x \in X$, $f \in C(X)$.
	\item The conditional expectation $E:C_\sigma(X,M) \rightarrow C(X)$ is defined by $E(a) = \tau \circ a$ for $a \in C_\sigma(X,M)$.
\end{itemize}
The axioms (T), (F) and (C) are satisfied.
\end{dfn} 

In \cite[Section 5]{Oz13}, Ozawa defines the fibre of a W$^*$-bundle $\M$ at $x \in X$ to be the image of $\M$ under the GNS representation $\pi_x:\M\rightarrow \H_x$ corresponding to the trace $a \mapsto E(a)(x)$. This trace induces a canonical faithful trace on the fibre of $\M$ at $x$. In the case of the trivial bundle $C_\sigma(X,M)$, it is more natural to use the evaluation map $\mathrm{eval}_x:C_\sigma(X,M) \rightarrow M$ than the GNS representation to define the fibre at $x$. Since both $*$-homomorphisms have kernel $\lbrace a \in C_\sigma(X,M): E(a^*a)(x) = 0 \rbrace$, the First Isomorphism Theorem gives us an isomorphism $\varphi$ such that the diagram
\begin{equation}
\xymatrix
{
	C_\sigma(X,M)  \ar[dr]^-{\pi_x} \ar[d]_-{\mathrm{eval}_x}  \\
	M  \ar[r]_-{\varphi}& \pi_x(C_\sigma(X,M)) \\
}	
\end{equation}
commutes. Hence, the two ways of defining fibres for a trivial bundle agree.	

In this paper, we find it most convenient to view the fibre of a general W$^*$-bundle at $x \in X$ as the quotient $\M/I_x$, where $I_x = \lbrace a \in \M: E(a^*a)(x) = 0 \rbrace$. We denote the fibre of $\M$ at $x$ by $\M_x$. We write $\tau_x$ for the induced faithful trace on this quotient and we write $a \mapsto a(x)$ for the canonical quotient map $\M \rightarrow \M_x$. Note that if $f \in C(X) \subseteq \M$, then $f - f(x)1_{\M} \in I_x$, so the image of $f$ in $\M_x$ is $f(x)1_{\M_x}$. This justifies the notation.      

We now fix a W$^*$-bundle $\M$ and show how the norms on the bundle relate to the corresponding norms on the fibres. 
\begin{prop}\label{SetTheoreticSections}
	The map
	\begin{align*}
		\Phi:\M &\rightarrow \prod_{x \in X} \M_x\\
		a &\mapsto (a(x))_{x \in X}
	\end{align*} 
	is an isometric $*$-homomorphism. In particular,
	\begin{equation}
		\|a\|_\M = \sup_{x \in X} \|a(x)\|_{\M_x}. \label{C*NormIsSup}
	\end{equation}
\end{prop}
\begin{proof}
	For each $x \in X$ the map $a \mapsto a(x)$ is a $*$-homomorphism. Hence, $\Phi$ is a $*$-homomorphism. Suppose $a(x) = 0$. Then $a \in I_x$ and so $E(a^*a)(x) = 0$. Hence, $a(x) = 0$ for all $x \in X$ implies that $E(a^*a) = 0$ and, consequently, $a = 0$ by the faithfulness axiom. Therefore, $\Phi$ is an injective $*$-homomorphism and thus isometric.
\end{proof}
\begin{prop} \label{Fibres2norm}
For fixed $a \in \M$, the map $x \mapsto \|a(x)\|_{2, \tau_x}$ is continuous. Furthermore, we have
\begin{equation} \label{Uniform2NormIsSup}
	\|a\|_{2,u} = \sup_{x \in X} \|a(x)\|_{2, \tau_x}.
\end{equation}
\end{prop}
\begin{proof}
	The proposition follows from the observation that $\|a(x)\|_{2, \tau_x} = E(a^*a)(x)^{1/2}$.
\end{proof}

We can now prove the key result that the fibres of a W$^*$-bundle are tracial von Neumann algebras. This result is due to Ozawa \cite[Theorem 11]{Oz13}, at least in the case where $X$ is metrisable. The proof given here avoids the use of Pedersen’s up-down theorem in Ozawa's proof by showing completeness of the unit ball via the argument in \cite[Proposition 10.1.12]{Di77}.  

\begin{thm}\label{thm:FibreAreVNAS}
	For each $x \in X$, $\M_x$ is a tracial von Neumann algebra.
\end{thm}
\begin{proof}
	Fix $x \in X$. We need to show that the closed unit ball $\lbrace b \in M_x: \|b\| \leq 1 \rbrace$ is complete with respect to the $\|\cdot\|_{2,\tau_x}$-norm. Let $(b_n) \subseteq \M_x$ be a sequence that satisfies $\|b_n\| \leq 1$ for all $n \in \mathbb{N}$ and is a Cauchy sequence with respect to the $\|\cdot\|_{2,\tau_x}$-norm on $\M_x$. Since a Cauchy sequence will converge to the limit of any convergent sub-sequence, we may assume that $\|b_{n+1} - b_n\|_{2, \tau_x} < \frac{1}{2^n}$ without loss of generality.

We shall construct a sequence $(a_n) \subseteq \M$ inductively such that
\begin{align}
	a_n(x) &= b_n, \\
	\|a_n\| &\leq 1, \\
	\|a_{n+1} - a_n\|_{2, u} &< \frac{1}{2^n} 
\end{align}
for all $n \in \mathbb{N}$. Recall that with C$^*$-algebras we may always lift elements from quotient algebras without increasing the norm. Let $a_1$ be any such lift of $b_1$. Suppose now that $a_1, \ldots, a_n$ have been defined and have the desired properties. Let $a_{n+1}'$ be any lift of $b_{n+1}$ with $\|a_{n+1}'\| \leq 1$. Since
\begin{equation}
	\|a_{n+1}'(x) - a_n(x)\|_{2, \tau_x} < \frac{1}{2^n},
\end{equation}
we can, by continuity, find an open neighbourhood $U$ of $x$ such that 
\begin{equation} \label{eqn: supU}
	\sup_{y \in U}\|a_{n+1}'(y) - a_n(y)\|_{2, \tau_y} < \frac{1}{2^n}.
\end{equation}
We then take a continuous function $f:X \rightarrow [0,1]$ such that $f(x) =  1$ and $f(X \setminus U) \subseteq \lbrace 0 \rbrace$, and set $a_{n+1} = fa_{n+1}' + (1-f)a_n$. We have that $a_{n+1}(x) = a_{n+1}'(x)$ and, using (\ref{C*NormIsSup}), we see that $\|a_{n+1}\| \leq 1$. Finally, we have that 
\begin{equation}
	\|a_{n+1}(y) - a_n(y)\|_{2,\tau_y} = |f(y)|\|a_{n+1}'(y) - a_n(y)\|_{2,\tau_y}
\end{equation}
for $y \in X$. By considering  the cases $y \in U$ and $y \in X \setminus U$ separately in (\ref{Uniform2NormIsSup}), we get that $\|a_{n+1} - a_n\|_{2,u} < \frac{1}{2^n}$. This completes the inductive definition of the sequence $(a_n)$. 

The sequence $(a_n)$ converges to some $a \in \M$ with $\|a\| \leq 1$ because the unit ball of $\M$ is complete in the $\|\cdot\|_{2,u}$-norm. Set $b = a(x)$. Then $(b_n)$ converges in $\|\cdot\|_{2,\tau_x}$-norm to $b$ by (\ref{Uniform2NormIsSup}).
\end{proof}

We now turn to the definition of the restriction of a W$^*$-bundle $\M$ over $X$ to a closed subset $Y$. The result will be a W$^*$-bundle $\M_Y$ over $Y$ together with a morphism of W$^*$-bundles $\M \rightarrow \M_Y$. The procedure is closely modelled on the definition of the fibres as quotients. Indeed, when $Y$ is a singleton $\lbrace x \rbrace$, the result is the fibre $\M_x$ viewed as a W$^*$-bundle over a one point space. 
 
\begin{dfn}
Let $\M$ be a W$^*$-bundle over the compact Hausdorff space $X$ with conditional expectation $E$, and let $Y$ be a closed subspace of $X$. The restriction of the W$^*$-bundle $\M$ to $Y$ is the C$^*$-algebra $\M_Y = \M/I_Y$, where $I_Y = \lbrace a \in \M: E(a^*a)(x) = 0 \; \forall x \in Y \rbrace$, together with the following embedding and conditional expectation: 
\begin{itemize}
\item The embedding of $C(Y) \cong C(X)/C_0(X \setminus Y)$ in the centre of $\M/I_Y$ is the induced embedding coming from the embedding of $C(X)$ in the centre of~$\M$. 
\item The conditional expectation $E_Y:\M/I_Y \rightarrow C(X)/C_0(X \setminus Y) \cong C(Y)$ is the induced conditional expectation coming from the conditional expectation $E:\M \rightarrow C(X)$.
\end{itemize}  
\end{dfn}   
\begin{prop} \label{RestrictionDetails}
The restriction $\M_Y$, as defined above, is a well-defined W$^*$-bundle. The canonical $*$-homomorphism $\M \rightarrow \M / I_Y$ defines a morphism of W$^*$-bundles.
\end{prop}
\begin{proof}
We first show that $I_Y$ is a norm-closed, two-sided ideal of $\M$. Since $I_Y = \bigcap_{x \in Y} I_x$, it is enough to note that each $I_x$ is a norm-closed, two-sided ideal. This is standard: $I_x$ is the kernel ideal of the trace $a \mapsto E(a)(x)$.

The induced embedding arises because $C(X) \cap I_Y = C_0(X \setminus Y)$. We have a commuting diagram 
\begin{equation}\label{EmbeddingCD}
\xymatrix
{
	\M \ar[r] & \M_Y  \\
	C(X) \ar[r] \ar[u] & C(X) / C_0(X \setminus Y), \ar[u]\\
}	
\end{equation}
where the vertical maps are the central embedding and the horizontal maps are the quotient maps.

Since $|E(a)(x)| \leq E(a^*a)(x)^{1/2}$ for all $x \in X$, we have $E(a) \in C_0(X \setminus Y)$ for all $a \in I_Y$. Hence there is a unital completely positive map $E_Y$ such that the diagram
\begin{equation}\label{ExpectationCD}
\xymatrix 
{
	\M \ar[r] \ar[d]_{E} & \M_Y \ar[d]_{E_Y} \\
	C(X) \ar[r]  & C(X) / C_0(X \setminus Y), \\
}	
\end{equation}
where the horizontal maps are the quotient maps, commutes. A diagram chase shows that $E_Y$ is a conditional expectation onto $C(X) / C_0(X \setminus Y) \subseteq \M_Y$. By the definition of $E$ it is straightforward to check (T) and (F) are satisfied for $E_Y$.

All that remains is to prove (C) for $\M_Y$. For this, we'll need to pass to fibres. Write $a \mapsto a|_Y$ for the canonical map $\M \rightarrow \M_Y$. Let $y \in Y$. Since $I_Y \subseteq I_y$, the fibre map $a \mapsto a(y)$ factors through the restriction map $a \mapsto a|_Y$, and we have a commuting diagram for the central embeddings
\begin{equation}
\xymatrix
{
	\M \ar[r] & \M_Y \ar[r] &\M_y  \\
	C(X) \ar[r] \ar[u] & C(Y) \ar[r]\ar[u] & \mathbb{C} \ar[u]\\
}	
\end{equation}
and a commuting diagram for the conditional expectations 
\begin{equation}
\xymatrix 
{
	\M \ar[r] \ar[d]_{E} & \M_Y \ar[r] \ar[d]_{E_Y} & \M_y \ar[d]_{\tau_y}  \\
	C(X) \ar[r]  & C(Y) \ar[r] &\mathbb{C}. \\
}	
\end{equation}

Hence, we can identify the fibre of $\M$ at $y$ with the fibre of $\M_Y$ at $y$. We obtain the following analogues of (\ref{C*NormIsSup}) and (\ref{Uniform2NormIsSup}):
\begin{align}
	\|a|_Y\| &= \sup_{y \in Y} \|a(y)\| \label{C*NormIsSupII} \\
	\|a|_Y\|_{2,u} &= \sup_{y \in Y} \|a(y)\|_{2, \tau_y}. \label{Uniform2NormIsSupII}
\end{align}

Let $(b_n) \subseteq \M_Y$ be a sequence that satisfies $\|b_n\| \leq 1$ for all $n \in \mathbb{N}$ and is a Cauchy sequence with respect to the $\|\cdot\|_{2,u}$-norm on $\M_Y$. We need to find $b \in \M_Y$ with $\|b\| \leq 1$ such that $(b_n)$ converges to $b$ in the $\|\cdot\|_{2,u}$-norm on $\M_Y$. Since a Cauchy sequence will converge to the limit of any convergent sub-sequence, we may assume that $\|b_{n+1} - b_n\|_{2, u} < \frac{1}{2^n}$ without loss of generality.

In the same way as in the proof of Theorem~\ref{thm:FibreAreVNAS}, we now inductively construct a sequence $(a_n) \subseteq \M$ such that 
\begin{align}
	a_n|_Y &= b_n, \\
	\|a_n\| &\leq 1, \\
	\|a_{n+1} - a_n\|_{2,u} &< \frac{1}{2^n}. 
\end{align}
The role of the point $x$ in the proof of Theorem~\ref{thm:FibreAreVNAS} is taken over by the compact set $Y$. We use compactness of $Y$ to obtain an open set $U \supseteq Y$ such that (\ref{eqn: supU}) holds and Urysohn's Lemma to obtain the continuous function $f \colon X \to [0,1]$ with $f(Y) = \{1\}$ and $f(X \setminus U) \subseteq \{0\}$ needed in the construction.

The sequence $(a_n)$ converges to some $a \in \M$ with $\|a\| \leq 1$ because the unit ball of $\M$ is complete in the $\|\cdot\|_{2,u}$-norm. We set $b = a|_Y$. The convergence of $(b_n)$ to $b$ follows by (\ref{Uniform2NormIsSupII}).
 
The morphism claim follows from the commuting diagrams (\ref{EmbeddingCD}) and (\ref{ExpectationCD}).
\end{proof}



We end this section with the definition of local triviality for a W$^*$-bundle. We note in particular that the isomorphism class of the fibres for a locally trivial bundle is locally constant. 
\begin{dfn}
	We say that a W$^*$-bundle $\M$ over $X$ is \emph{locally trivial} if every $x \in X$ has a closed neighbourhood $Y$ such that $\M_Y$ is isomorphic to a trivial bundle over $Y$. 
\end{dfn}

\section{The Topological Bundle}
In this section we shall show how to combine the fibres of a W$^*$-bundle $\M$ to produce a bundle $(B,p)$ in the sense of \cite[Chapter 2, Section 13.1]{FD88}. The W$^*$-bundle, more precisely its section algebra, can be recovered as the collection of bounded, continuous sections of $(B,p)$. This builds on known results in the context of continuous fields of Hilbert spaces \cite[Section 1.2]{Di63} and Banach bundles \cite[Chapter 2, Section 13.4]{FD88}. We begin by recalling the general definition of a bundle from \cite[Chapter 2, Section 13.1]{FD88}.

\begin{dfn}
A bundle over a Hausdorff topological space $X$ is a pair $(B, p)$ where $B$ is a Hausdorff topological space and $p:B \rightarrow X$ is a continuous, open surjection. The fibre at $x \in X$ is the set $p^{-1}(x)$. 
\end{dfn}

By abuse of notation, we shall often speak of the bundle $B$ instead of the bundle $(B,p)$. We shall employ $B_x$ as an alternative notation for the fibre at $x$. We denote the set $\lbrace (b_1, b_2) \in B \times B: p(b_1) = p(b_2) \rbrace$ by $B \times_p B$ and endow it with the subspace topology coming from $B \times B$.  

We now describe the additional structure necessary for a bundle to be a topological bundle of tracial von Neumann algebras. For this definition it is best to view tracial von Neumann algebras abstractly as C$^*$-algebras with a tracial state such that the unit ball is complete with respect to the 2-norm (see \cite[Lemma A.3.3]{Si08}); an isomorphism of tracial von Neumann algebras is a trace preserving isomorphism of the C$^*$-algebras.  

\begin{dfn} \label{dfn:TopologicalBundle}
A topological bundle of tracial von Neumann algebras over the Hausdorff space $X$ is a bundle $(B,p)$ over $X$, together with operations, norms and traces making each fibre $B_x$ a tracial von Neumann algebra and satisfying the axioms listed below: 
\begin{itemize}
	\item[(i)] Addition, viewed as a map $B \times_p B \rightarrow B$, is continuous.
	\item[(ii)] Scalar multiplication, viewed as a map $\mathbb{C} \times B \rightarrow B$, is continuous.
	\item[(iii)] The involution, viewed as a map $B \rightarrow B$, is continuous.
	\item[(iv)] The map $X \rightarrow B$ which sends $x$ to the to the additive identity $0_x$ of $B_x$ is continuous and so is the analogous map $X \rightarrow B$ which sends $x$ to the to the multiplicative identity $1_x$ of $B_x$.
	\item[(v)] The map $\tau:B \rightarrow \mathbb{C}$ obtained by combining the traces on each fibre is continuous and so is the map $\|\cdot\|_2:B \rightarrow \mathbb{C}$ arising from combining the 2-norms from each fibre.
	\item[(vi)] A net $(b_\lambda) \subseteq B$ converges to $0_x$ whenever $p(b_\lambda) \rightarrow x$ and $\|b_\lambda\|_2 \rightarrow 0$.
\end{itemize}
For the last two axioms the map $\|\cdot\|:B \rightarrow [0, \infty]$ obtained by combining the C$^*$-norms from each fibre plays an auxillary role. We shall write $B_{\leq r}$ for the subspace $\lbrace b \in B: \|b\| \leq r \rbrace$ of $B$ for $r > 0$. 
\begin{itemize}
	\item[(vii)] Multiplication, viewed as a map $B \times_p B \rightarrow B$, is continuous on $\|\cdot\|$-bounded subsets. 
	\item[(viii)] The restriction $p|_{B_{\leq 1}}: B_{\leq 1} \rightarrow X$ is open.
\end{itemize}


We say that two topological bundles of tracial von Neumann algebras $(B_i, p_i)$  for $i=1,2$ are isomorphic if there are homeomorphisms $\psi$ and $\varphi$ such that the diagram 
\begin{equation}\label{Isomorphism} 
\xymatrix
{
	B_1 \ar[r]^{\varphi} \ar[d]_{p_1} & B_2 \ar[d]^{p_2} \\
	X_1 \ar[r]^{\psi} & X_2 \\
}	
\end{equation}
commutes and, for each $x_1 \in X_1$, $\varphi|_{p_1^{-1}(x_1)}:p_1^{-1}(x_1) \rightarrow p_2^{-1}(\psi(x_1))$ is an isomorphism of tracial von Neumann algebras. 
\end{dfn}

\begin{rem}
	The axioms are modelled on the definition of a Banach bundle given in \cite[Chapter 2, Section 13.4]{FD88}. Note, however, that the fibres are not complete in the $\|\cdot\|_2$-norm. We only have $\|\cdot\|_2$-norm completeness of the $\|\cdot\|$-norm closed unit ball. 
\end{rem}

The basic example of a topological bundle of tracial von Neumann algebras is $(X \times M, \pi_1)$ , where $X$ is a Hausdorff space, $M$ is a tracial von Neumann algebra, the topology on $X \times M$ is the product of the topology of $X$ and the 2-norm topology on $M$, and $\pi_1: X \times M \rightarrow X$ is the projection onto the first coordinate. This is the trivial bundle over $X$ with fibre $M$. We can now define local triviality for topological bundles of tracial von Neumann algebras.

\begin{dfn}\label{dfn:TopBundleLocTrivial}
	Let $(B, p)$ be a topological bundle of tracial von Neumann algebras over the Hausdorff space $X$. We say $(B, p)$ is locally trivial if every $x \in X$ has an open neighbourhood $U$ such that $(p^{-1}(U), p|_{p^{-1}(U)})$ is isomorphic to a trivial bundle over $U$.
\end{dfn}   
\begin{rem}
Note that when $X$ is compact Hausdorff, we can work with closed neighbourhoods in place of open neighbourhoods.
\end{rem}

Let $\M$ be a W$^*$-bundle over the compact Hausdorff space $X$. Set $B = \bigsqcup_{x \in X} \M_x$ and define  $p:B \rightarrow X$ by $p(b) = x$ whenever $b \in \M_x$. Note that, for each $x \in X$, the fibre $p^{-1}(x)$ can be identified with $\M_x$ and, therefore, endowed with operations, a norm and a trace that make it a tracial von Neumann algebra. In the following proposition, we define a topology on $B$ so that $(B,p)$ is a topological bundle of tracial von Neumann algebras. We then check that isomorphic W$^*$-bundles give rise to isomorphic topological bundles.
\begin{prop}\label{TopologyDetails}
	Let $\M$ be a W$^*$-bundle over the compact Hausdorff space $X$. Set $B = \bigsqcup_{x \in X} \M_x$ and define  $p:B \rightarrow X$ by $p(b) = x$ whenever $b \in \M_x$. For $a \in \M$, $\epsilon > 0$ and $U$ open in $X$, we set $V(a, \epsilon, U) = \lbrace b \in B: p(b) \in U, \|a(p(b)) - b\|_2 < \epsilon \rbrace$. 
\begin{itemize}
	\item[(a)] The collection $\mathcal{B}$ of all such $V(a, \epsilon, U)$ form a basis for a topology on $B$. Moreover, if $b \in B$ and $a \in \M$ is chosen with $a(p(b)) = b$, then the collection of $V(a, \epsilon, U)$ as $\epsilon$ ranges over positive reals and $U$ ranges over a neighbourhood basis of $p(b)$ is a neighbourhood basis of $b$.  
	\item[(b)] When $B$ is endowed with the topology generated by $\mathcal{B}$, $(B,p)$ is a topological bundle of tracial von Neumann algebras.
\end{itemize}	 
\end{prop}
\begin{proof}
(a) Given $b \in B$, let $x = p(b)$, so $b \in \M_x$. Let $a \in \M$ be a lift of $b$. Then, for any open neighbourhood $U$ of $x$ and $\epsilon > 0$, $b \in V(a, \epsilon, U)$. Therefore, $\bigcup_{V \in \mathcal{B}} V = B$.
	
	Suppose $b \in V(a_1, \epsilon_1, U_1) \cap V(a_2, \epsilon_2, U_2)$. Set $x = p(b)$, and let $a \in \M$ be a lift of $b \in \M_x$.  We have $x \in U_1 \cap U_2$ and 
	\begin{align}		
		\delta_1 := \|a(x) - a_1(x)\|_2 &< \epsilon_1 \\
		\delta_2 := \|a(x) - a_2(x)\|_2 &< \epsilon_2 \nonumber
	\end{align}
	Choose, by continuity, an open set $U$ such that $x \in U \subseteq U_1 \cap U_2$, and such that 
	\begin{align}
		\|a(x') - a_1(x')\|_2 &< \frac{\epsilon_1 + \delta_1}{2}  \\
		\|a(x') - a_2(x')\|_2 &< \frac{\epsilon_2 + \delta_2}{2}.\nonumber
	\end{align}
	for all $x' \in U$.
	Set $\epsilon = \min(\frac{\epsilon_1 - \delta_1}{2}, \frac{\epsilon_2 - \delta_2}{2})$. Now, if $b' \in V(a, \epsilon, U)$, then, for $i\in \{1,2\}$, $x' := p(b') \in U_i$ and
	\begin{align}
		\|a_i(x') - b' \|_2 &\leq \|a_i(x') - a(x')\|_2 + \|a(x') - b'\|_2 \\
		&< \frac{\epsilon_i + \delta_i}{2} + \epsilon \nonumber\\
		&\leq \epsilon_i. \nonumber
	\end{align}
	So, $b' \in V(a_i, \epsilon_i, U_i)$. Hence, $b \in V(a, \epsilon, U) \subseteq V(a_1, \epsilon_1, U_1) \cap V(a_2, \epsilon_2, U_2)$. 

	This proves that $\mathcal{B}$ does form the basis for a topology on $B$, and also gives the required neighbourhood basis for $b \in B$. 
	
(b) The topology defined by $\mathcal{B}$ is easily seen to be Hausdorff. Let $U$ be open in $X$. Let $b \in p^{-1}(U)$ with $x = p(b)$. Choose $a \in \M$ with $a(x) = b$. Then $b \in V(a, 1, U) \subseteq p^{-1}(U)$. So $p^{-1}(U)$ is open in $B$. Hence $p$ is continuous. It is clearly surjective. We now check the axioms of Definition \ref{dfn:TopologicalBundle} in turn, noting that a simple scaling argument shows that axioms (ii) and (viii) imply that the map $p:B \rightarrow X$ is open. 

(i) Let $b_1, b_2 \in B$ with $x = p(b_1) = p (b_2)$. Let $a_1,a_2 \in \M$ be lifts of $b_1, b_2 \in \M_x$. A basic open neighbourhood of $b_1 + b_2$ has the form $V(a_1 + a_2, \epsilon, U)$ for some $\epsilon > 0$ and open neighbourhood $U$ of $x$. Let $b_1' \in V(a_1, \tfrac{\epsilon}{2}, U)$ and $b_2' \in V(a_2, \tfrac{\epsilon}{2}, U)$, and suppose $x' = p(b_1') = p(b_2') \in U$. We have
\begin{align}
	\|(a_1(x') + a_2(x')) - (b_1' + b_2')\|_2 &\leq \|a_1(x') - b_1'\|_2 + \|a_2(x') - b_2'\|_2 \\
	&< \frac{\epsilon}{2} + \frac{\epsilon}{2}\nonumber\\
	&= \epsilon.\nonumber
\end{align}
So, $b_1' + b_2' \in  V(a_1 + a_2, \epsilon, U)$.
 
(ii) Let $\lambda \in \mathbb{C}$ and $b \in B$ with $x = p(b)$. Choose $a \in \M$ with $a(x) = b$. A basic neighbourhood of $\lambda b$ has the form $V(\lambda a, \epsilon, U)$ for some $\epsilon > 0$ and some open neighbourhood $U$ of $x$ in $B$. Set $K = \max(\|a\|_{2,u}, |\lambda|) + 1$ and $\delta = \min(\tfrac{\epsilon}{2K},1)$. Let $|\lambda' - \lambda| < \delta$ and $b' \in V(a, \delta, U)$ with $x' = p(b')$. Then  
\begin{align}
	\|\lambda'b' - \lambda a(x')\|_2 &\leq |\lambda'|\|b' - a(x')\|_2 + |\lambda' - \lambda|\|a(x')\|_2 \\
	&\leq (|\lambda| + 1)\|b' - a(x')\|_2 + |\lambda' - \lambda|\|a(x')\|_2 \nonumber\\
	&< K\delta + \delta K \nonumber \\
	&\leq \epsilon.\nonumber
\end{align}

(iii) This follows from the observation $V(a, \epsilon, U)^* = V(a^*, \epsilon, U)$.
 
(iv) For the continuity of the map $x \mapsto 1_x$ it suffices to observe that the open set $V(1, \epsilon, U)$ has preimage $U$ under this map for any $\epsilon > 0$. The continuity of $x \mapsto 0_x$ is similar. 

(v) We show the continuity of $\|\cdot\|_2$ on $B$. Continuity of $\tau$ then follows by the polarisation identity together with the continuity of $x \mapsto 1_x$. Let $b \in B$ and $a \in \M$, $x \in X$ be such that $a(x) = b$. Let $\epsilon > 0$. By Proposition \ref{Fibres2norm}, the map $y \mapsto \|a(y)\|_2$ is continuous. Hence, there is an open set $U \ni x$ such that 
\begin{equation}
	\Big|\|a(y)\|_2 - \|a(x)\|_2\Big| < \frac{\epsilon}{2}
\end{equation} 
for all $y \in U$. Let $b' \in V(a, \tfrac{\epsilon}{2}, U)$.  Writing $x' = p(b')$, we have
\begin{align}
	\Big|\|b'\|_2 - \|b\|_2\Big| &\leq \Big|\|b'\|_2 - \|a(x')\|_2\Big| + \Big|\|a(x')\|_2 - \|a(x)\|_2\Big| \\
	&< \frac{\epsilon}{2} + \frac{\epsilon}{2} \nonumber\\
	&= \epsilon. \nonumber \qedhere
\end{align}

(vi) This follows from the fact that a basic open neighbourhood of $0_x$ has the form $V(0,\epsilon, U)$ for some $\epsilon > 0$ and some open neighbourhood $U$ of $x$ in $X$.

(vii) Fix $K > 0$. Let $b_1, b_2 \in B$ with $\|\cdot\|$-norm bounded by $K$. Suppose $x = p(b_1) = p (b_2)$. Let $a_1,a_2 \in \M$ be norm-preserving lifts of $b_1, b_2 \in \M_x$. A basic open neighbourhood of $b_1b_2$ has the form $V(a_1a_2, \epsilon, U)$ for some $\epsilon > 0$ and open neighbourhood $U$ of $x$. Let $b_1' \in V(a_1, \tfrac{\epsilon}{2K}, U)$ and $b_2' \in V(a_2, \tfrac{\epsilon}{2K}, U)$. Assume $b_1'$ and $b_2'$ are $\|\cdot\|$-norm bounded by $K$, and that $x' = p(b_1') = p(b_2') \in U$. We have
\begin{align}
	\|a_1(x')a_2(x') - b_1'b_2'\|_2 &\leq \|a_1(x')\|\|a_2(x') - b_2'\|_2 + \|a_1(x') - b_1'\|_2\|b_2'\| \\
	&\leq K\|a_2(x') - b_2'\|_2 + K\|a_1(x') - b_1'\|_2\nonumber\\
	&< K\left(\frac{\epsilon}{2K} + \frac{\epsilon}{2K}\right)\nonumber\\
	&= \epsilon.\nonumber
\end{align}
So, $b_1'b_2' \in  V(a_1a_2, \epsilon, U)$.

(viii) Let $W$ be open in $B$ with $W \cap B|_{\leq 1} \neq \emptyset$. Let $x \in p(W \cap B|_{\leq 1})$. Choose $b \in W \cap B|_{\leq 1}$ such that $p(b) = x$. Lift $b \in \M_x$ to an element $a \in \M$ of the same norm. The open set $W$ contains a basic open neighbourhood of the form $V(a, \epsilon, U)$, where $\epsilon > 0$ and $U$ is a neighbourhood of $x$ in $X$. Hence, for all $x' \in U$, it follows that $a(x') \in W$ and $\|a(x')\|_{\M_{x'}} \leq 1$. Therefore $U \subseteq p(W \cap B|_{\leq 1})$ and so $p|_{B_{\leq 1}}: B_{\leq 1} \rightarrow X$ is open.
\end{proof}

\begin{prop} \label{Iso1}
	Let $\M_i$ be a W$^*$-bundle over $X_i$ with conditional expectation $E_i$ for $i=1,2$. Let $(B_i, p_i)$ be the corresponding topological bundle of tracial von Neumann algebras for $i=1,2$. If the W$^*$-bundles are isomorphic then the topological bundles are isomorphic.
\end{prop}
\begin{proof}
Assume $\alpha:\M_1 \rightarrow \M_2$ is an isomorphism of the W$^*$-bundles. Then $\alpha$ restricts to an isomorphism $C(X_1) \rightarrow C(X_2)$, so induces a homeomorphism $\alpha^t:X_2 \rightarrow X_1$. Since $E_2(\alpha(a))(x_2) = \alpha(E_1(a))(x_2) = E_1(a)(\alpha^t(x_2))$, $\alpha$ induces an isomorphism between the fibres $(\M_1)_{\alpha^t(x_2)}$ and $(\M_2)_{x_2}$ for each $x_2 \in X_2$. Combining all these isomorphisms, we get a bijection $\varphi:B_1 \rightarrow B_2$ such that (\ref{Isomorphism}) holds with $\psi = (\alpha^t)^{-1}$. By considering the basic open neighbourhoods in $B_1$ and $B_2$, we see that $\varphi$ is a homeomorphism. Indeed, $\varphi(V_{\M_1}(a, \epsilon, U)) = V_{\M_2}(\alpha(a), \epsilon, \psi(U))$ for all $a \in \M_1$, $\epsilon > 0$, and $U$ open in $X_1$.
\end{proof}

In the other direction, given a topological bundle of tracial von Neumann algebras over a compact Hausdorff space, we can define a W$^*$-bundle. This comes from considering sections.

\begin{dfn}\label{DfnSections}
	Let $(B, p)$ be a topological bundle of tracial von Neumann algebras over $X$. A section of $(B,p)$ is a map $s:X \rightarrow B$ such that $p \circ s = \mathrm{id}_X$. A section is said to be bounded if $\sup_{x \in X} \|s(x)\| < \infty$. 
\end{dfn}
\begin{rem}
Since it is not required that $\|\cdot\|$ be continuous on $B$, continuous sections $s:X \rightarrow B$ are not automatically bounded in the sense of Definition \ref{DfnSections} even when $X$ is compact. 
\end{rem}

Let $(B, p)$ be a topological bundle of tracial von Neumann algebras over the compact Hausdorff space $X$. The set of bounded sections of $(B, p)$ endowed with fibrewise-defined operations and the uniform norm $\|s\| = \sup_{x \in X} \|s(x)\|$ is a C$^*$-algebra isomorphic to the product $\prod_{x \in X} p^{-1}(x)$. Since the fibres are tracial von Neumann algebras, the uniform 2-norm $\|s\|_{2,u} = \sup_{x \in X}\|s(x)\|_2$ is complete when restricted to the closed unit ball in uniform norm. Let $\M$ be the collection of bounded, continuous sections. Axioms (i)--(viii) ensure that $\M$ is a unital $^*$-subalgebra. The following proposition shows that continuity of sections is preserved under uniform-2-norm limits and, a fortiori, under uniform-norm limits. Therefore, $\M$ inherits the completeness properties of the algebra of bounded sections, in particular $\M$ is a C$^*$-algebra. 
\begin{prop}\label{prop:Completeness}
Let $s_n:X \rightarrow B$ be bounded, continuous sections of a topological bundle $(B, p)$ of tracial von Neumann algebras. Assume that the sequence $(s_n)$ converges in uniform 2-norm to the bounded section $s:X \rightarrow B$. Then $s$ is continuous.		
\end{prop}
\begin{proof}
	Let $W$ be open in $X$. Then, by hypothesis, $p(W \cap B_{\leq 1})$ is open in $X$. By scaling, we get that $p(W \cap B_{\leq r})$ is open for all $r > 0$. Taking unions, $p(W)$ is open. Hence $p$ is an open map. We can now apply the argument in \cite[Chapter 2, Section 13.13]{FD88}, which makes no reference to the completeness of the fibres.  
\end{proof}
The additional data for a W$^*$-bundle over $X$ with section algebra $\M$ can now be easily defined and the axioms verified. We identify $f \in C(X)$ with the scalar valued section $x \mapsto f(x)1_x$. Such scalar valued sections are clearly bounded and are continuous since scalar multiplication and the section $x \mapsto 1_x$  are continuous. This gives an inclusion $C(X) \subseteq Z(\M)$. We define $E:\M \rightarrow C(X)$ by $s \mapsto \tau \circ s$. This is a conditional expectation from $\M$ onto the image of $C(X)$ in $\M$ and induces the uniform 2-norm on $\M$. Axiom (C) follows from Proposition \ref{prop:Completeness}. Axioms (T) and (F) follow fibrewise from the corresponding properties of a faithful trace. 

As before, we check that our construction is compatible with our notions of isomorphism.   
\begin{prop} \label{Iso2}
Let $(B_i, p_i)$ be a topological bundle of tracial von Neumann algebras over the compact Hausdorff space $X_i$ for $i=1,2$. Let $\M_i$ be the  W$^*$-bundle over $X_i$ with conditional expectation $E_i$ that comes from $(B_i, p_i)$. If the topological bundles are isomorphic then the W$^*$-bundles are isomorphic.
\end{prop}
\begin{proof}
If the topological bundles are isomorphic and $\varphi$ and $\psi$ are as in (\ref{Isomorphism}) then $s \mapsto \varphi \circ s \circ \psi^{-1}$ defines a bijection between the bounded, continuous section of $p_1:B_1 \rightarrow X_1$ and those of $p_2:B_2 \rightarrow X_2$, that is a map $\alpha:\M_1 \rightarrow \M_2$. 

Since for each $x_1 \in X_1$, $\varphi|_{p_1^{-1}(x_1)}:p_1^{-1}(x_1) \rightarrow p_2^{-1}(\psi(x_1))$ is an isomorphism of tracial von Neumann algebras, $\alpha$ is a $*$-homomorphism of C$^*$-algebras. Furthermore, the following computations show that $\alpha$ is a morphism of W$^*$-bundles. Firstly, let $f_1 \in C(X_1) \subseteq Z(\M_1)$ and $x_2 \in X_2$. Then 
\begin{align}
	\alpha(f_1)(x_2) &= \varphi(f_1(\psi^{-1}(x_2))1_{\psi^{-1}(x_2)}) \\
		&= f_1(\psi^{-1}(x_2))1_{x_2},\nonumber
\end{align}
so $\alpha(f_1) = f_1 \circ \psi^{-1} \in C(X_2) \subseteq Z(\M_2)$. Secondly, let $s \in \M_1$ and $x_2 \in X$. Then 
\begin{align}
	E_2(\alpha(s))(x_2) &= \tau_{p_2^{-1}(x_2)}(\alpha(s)(x_2)) \\
	 &= \tau_{p_2^{-1}(x_2)}(\varphi(s(\psi^{-1}(x_2)))) \nonumber\\
	 &= \tau_{p_1^{-1}(\psi^{-1}(x_2))}(s(\psi^{-1}(x_2))) \nonumber\\
	 &= E_1(s)(\psi^{-1}(x_2))\nonumber\\
	 &= \alpha(E_1(s))(x_2),\nonumber 
\end{align}
so $E_2 \circ \alpha = \alpha \circ E_1$.
\end{proof}

We now investigate the inverse nature of the two constructions considered in the section. One direction is essentially \cite[Theorem 11]{Oz13}. The other direction reduces to the question of whether we can construct a bounded, continuous section through any point of the topological bundle.

\begin{thm} \label{Taka11}
Let $\M$ be a W$^*$-bundle over the compact Hausdorff space $X$. Let $(B,p)$ be the topological bundle constructed from $\M$. 
\begin{itemize}
\item[(a)] For each $a \in \M$, the map $s_a:X \rightarrow B$ given by $x \mapsto a(x) \in \M_x$ defines a bounded, continuous section of $(B,p)$. \item[(b)] Every bounded, continuous section of $(B, p)$ has the form $s_a$ for some $a \in \M$. 
\item[(c)] The map $a \mapsto s_a$ is an isomorphism between the W$^*$-bundle $\M$ and the W$^*$-bundle constructed from $(B, p)$. 
\end{itemize}
\end{thm}
\begin{proof}
(a) Let $a \in \M$. By construction $s_a$ is a section of $(B,p)$. We have $\|a(x)\|_{\M_x} \leq  \|a\|_{\M}$ for all $x \in X$, so the section $s_a$ is bounded. Let $W$ be open in $B$ and $x \in s_a^{-1}(W)$. Then $s_a(x) = a(x) \in W$. By Proposition \ref{TopologyDetails}(a), there exists $\epsilon > 0$ and an open neighbourhood $U$ of $x$ in $X$ such that $a(x) \in V(a, \epsilon, U) \subseteq W$. It follows that $x \in U \subseteq s_a^{-1}(W)$. Hence, $s_a$ is continuous. 

(b) Assume $s:X \rightarrow B$ is a continuous and bounded section. Let $x_0 \in X$ and $\epsilon > 0$. Choose $a_0 \in \M$ such that $a_0(x_0) = s(x_0)$. Since the function $x \mapsto \|s(x) - a_0(x)\|_2$ is continuous, there is a neighbourhood $U$ of $x_0$ such that 
\begin{equation}
	\sup_{x \in U} \|s(x) - a_0(x)\|_2 < \epsilon.
\end{equation}
By \cite[Theorem 11]{Oz13}, there exists $a \in \M$ such that $a(x) = s(x)$ for all $x \in X$.

(c) The map $a \mapsto s_a$ is a unital homomorphism of C$^*$-algebras. It is injective by Proposition \ref{SetTheoreticSections} and surjective by (b). For $f \in C(X) \subseteq Z(\M)$, $s_f$ is the scalar section $x \mapsto f(x)1_x$ (see the discussion preceding Proposition \ref{SetTheoreticSections}) and, for arbitrary $a \in \M$ and $x \in X$, $\tau(s_a(x)) = \tau_x(a(x)) = E(a)(x)$. Therefore $a \mapsto s_a$ is an isomorphism of W$^*$-bundles.  
\end{proof}

\begin{thm} \label{TheOtherOne}
	Let $(B,p)$ be a topological bundle of tracial von Neumann algebras over the compact Hausdorff space $X$. Let $\M$ be the W$^*$-bundle defined by considering bounded, continuous sections of $(B, p)$. Let $(\widetilde{B}, \widetilde{p})$ be the topological bundle constructed from the fibres of $\M$. Assume that, for all $b \in B$, there is $s \in \M$ with $s(p(b)) = b$. Then the topological bundles $(B,p)$ and $(\widetilde{B}, \widetilde{p})$ are isomorphic.
\end{thm}
\begin{proof}
	Write $E$ for the conditional expectation of $\M$. For each $x \in X$, consider the evaluation map $\varphi_x:\M \rightarrow p^{-1}(x)$ given by $s \mapsto s(x)$. This is a homomorphism of C$^*$-algebras and, by our assumption, it is surjective. Since $\tau(s(x)) = E(s)(x)$ for all $s \in \M$ and the trace on $p^{-1}(x)$ is faithful, we get an induced isomorphism of tracial von Neumann algebras $\overline{\varphi}_x: \M_x \rightarrow p^{-1}(x)$. Combining all such maps, we get a bijection $\varphi:\widetilde{B} \rightarrow B$, such that the diagram
\begin{equation}
\xymatrix
{
	\widetilde{B} \ar[r]^{\varphi} \ar[d]_{\widetilde{p}} & B \ar[d]^{p} \\
	X \ar[r]^{\mathrm{id}_X} & X \\
}	
\end{equation}
commutes. It remains to show that $\varphi$ is a homeomorphism. Note that, via our convention of writing $s(x)$ for the image of $s \in \M$ in $\M_x$, $\varphi$ can be viewed as the identity map on $B$. Thus proving that $\varphi$ is a homeomorphism amounts to showing that the topology on $B$, satisfying the axioms for a topological bundle, has a basis consisting of the sets $V(s, \epsilon, U) = \lbrace b \in B: p(b) \in U, \|s(p(b)) - b\|_2 < \epsilon \rbrace$ for $s \in \M$, $\epsilon > 0$ and $U$ open in $X$. 

Each such set $V(s, \epsilon, U)$ is open in $B$ because the axioms for a topological bundle ensure that the map $F \colon B \rightarrow \mathbb{R} \times X$ given by $b \mapsto (\|s(p(b)) - b\|_2, p(b))$ is continuous. We complete the proof by showing that the set of all such $V(s, \epsilon, U)$ contains a neighbourhood basis for each point of $B$. Axiom (ii) for topological bundles gives that $V(0, \epsilon, U)$ as $\epsilon$ ranges over the positive reals and $U$ ranges over a neighbourhood basis for $x \in X$ form a neighbourhood basis for $0_x$. Let $b_0 \in B$ and $s_0$ be a bounded, continuous section with $s_0(p(b_0)) = b_0$. Since the map $G:B \rightarrow B$ given by $b \mapsto s_0(p(b)) - b$ is a homeomorphism of $B$, we see that $V(s_0, \epsilon, U)$ as $\epsilon$ ranges over the positive reals and $U$ ranges over a neighbourhood basis for $p(b_0)$ form a neighbourhood basis for $b_0$. 
\end{proof}
\begin{rem}
	In all the topological bundles $(B, p)$ that we consider in this paper, the assumption that there is a bounded, continuous section through every point of the bundle space $B$ will be satisfied. Indeed it holds by construction for the topological bundles coming from W$^*$-bundles. It is also clear when $(B, p)$ is a locally trivial fibre bundle over a compact Hausdorff space with fibre a fixed tracial von Neumann algebra $M$, since such bundles look locally like the projection map $X \times M \rightarrow X$. 
\end{rem}


We observe that the topological bundle corresponding to a trivial W$^*$-algebra $C_\sigma(X, M)$, where $M$ is a fixed tracial von Neumann algebra and $X$ is a compact Hausdorff space, is $(X \times M, \pi_1)$, where the topology on $X \times M$ is the product of the topology of $X$ and the 2-norm topology on $M$ and $\pi_1: X \times M \rightarrow X$ is the projection onto the first coordinate. Thus, the notion of triviality for a topological bundle of tracial von Neumann algebras matches up with that for a W$^*$-bundle. We show below that the notions of restriction to a closed subset also match up and, therefore, so do the natural notions of local triviality.

\begin{prop} \label{Restr}
	Let $X$ be a compact Hausdorff space and $Y$ a closed subset. 
	\begin{itemize}
		\item[(a)] Let $\M$ be a W$^*$-bundle over $X$ and $(B,p)$ the corresponding topological bundle. Let $(B_Y, p_Y)$ be the topological bundle corresponding to the W$^*$-bundle $\M_Y$. There exists a homeomorphism $\varphi$ such that the diagram  
		\begin{equation*}
		\xymatrix
		{
			B_Y \ar[r]^{\varphi} \ar[d]_{p_Y} & p^{-1}(Y) \ar[d]^{p|_{p^{-1}(Y)}} \\
			Y \ar[r]^{\mathrm{id}_Y} & Y \\
		}	
		\end{equation*}
		commutes, which induces an isomorphism of tracial von Neumann algebras in each fibre.
		\item[(b)] Let $(B, p)$ be a topological bundle of tracial von Neumann algebras over $X$. Let $\M$ be the W$^*$-bundle arising from bounded, continuous sections of $(B,p)$. Then $\M_Y$ is isomorphic to the W$^*$-bundle $\widetilde{\M}$ of bounded, continuous sections of $(p^{-1}(Y),p|_{p^{-1}(Y)})$. 
	\end{itemize}
\end{prop}
\begin{proof}
	(a) For $y \in Y$, the fibre $(\M_Y)_y$ of $\M_Y$ can be identified with the fibre $\M_y$ of $\M$, via the map $a|_Y(y) \mapsto a(y)$. Combining all these maps, we obtain a bijection $\varphi$ such that the diagram commutes. Considering basic open neighbourhoods we see that $\varphi$ is a homeomorphism. Indeed, $\varphi(V_{\M_Y}(a|_Y, \epsilon, U \cap Y)) = V_{\M}(a, \epsilon, U) \cap p^{-1}(Y)$ for all $a \in \M$, $\epsilon > 0$, and $U$ open in $X$.   
	
	(b) Write $E$ for the conditional expectation of $\M$ and $\widetilde{E}$ for the conditional expectation on $\widetilde{\M}$. Restricting a bounded, continuous section $s:X \rightarrow B$ of $p$ to $Y$ gives a continuous bounded section of $(p^{-1}(Y),p|_{p^{-1}(Y)})$. This defines a homomorphism of C$^*$-algebras $\M \rightarrow \widetilde{\M}$. The kernel of this homomorphism is the ideal $I_Y = \lbrace s \in \M: E(s^*s)(y) = 0 \; \forall y \in Y \rbrace$. So we get an induced isometric homomorphism of C$^*$-algebras $\alpha:\M_Y \rightarrow \widetilde{\M}$. This homomorphism restricts to the identity map on the central copies of $C(Y)$ in $\M_Y$ and $\widetilde{\M}$, and the diagram
\begin{equation} 
\xymatrix
{
	\M_Y \ar[r]^{\alpha} \ar[d]_{E_Y} & \widetilde{\M} \ar[d]_{\widetilde{E}} \\
	C(Y) \ar[r]^{\mathrm{id}} & C(Y) \\
}
\end{equation}  
commutes. In particular, $\alpha$ preserves the uniform 2-norm. The argument to show that $\alpha$ is surjective has two parts. First, using a partition of unity argument as in \cite[Lemma 10.1.11]{Di77}, one shows that, for any continuous section $s: Y \rightarrow B$ with $\|s(y)\| \leq 1$ for all $y \in Y$ and any $\epsilon > 0$, there is a bounded, continuous section $\overline{s}:X \rightarrow B$ with $\|\overline{s}(x)\| \leq 1$ for all $x \in X$ and $\|s(y) - \overline{s}(y)\|_2 < \epsilon$. This implies that the $\|\cdot\|$-norm closed unit ball of $\M_Y$ has $\|\cdot\|_{2,u}$-dense image in the $\|\cdot\|$-norm closed unit ball of $\widetilde{M}$. The completeness of the $\|\cdot\|$-norm closed units balls in $\|\cdot\|_{2,u}$-norm then implies that $\alpha$ is surjective.
\end{proof}

\section{Locally Trivial Bundles}

In this section, we prove our main result: a locally trivial W$^*$-bundle with all fibres isomorphic to the hyperfinite II$_1$ factor $\RR$ is trivial. 

In fact, the only property of the II$_1$ factor $\mathcal{R}$ that we shall need is that its automorphism group is contractible. We begin, therefore, with a brief discussion of possible topologies on the automorphism group $\Aut{M}$ of a tracial von Neumann algebra $M$, and note that in the factor case they coincide. 

\begin{dfn} \label{dfn:top_on_Aut}  \cite[Definition 3.4]{Ha75}
Let $M$ be a von Neumann algebra with a faithful, normal trace $\tau \colon M \to \C$. Let $\mathcal{B}_*(M)$ be the set of bounded $\sigma$-weakly continuous operators on~$M$.
\begin{itemize}
	\item The \emph{$u$-topology} on $\mathcal{B}_*(M)$ is the topology generated by the seminorms
\(
	\lVert T \rVert^u_{\varphi} = \lVert \varphi \circ T \rVert 
\)
for all $\varphi \in M_*$.
	\item The \emph{$p$-topology} on $\mathcal{B}_*(M)$ is defined via the seminorms $\lVert T \rVert^p_{\varphi,a} = \lvert (\varphi \circ T)(a) \rvert$ for all $a \in M$ and $\varphi \in M_*$.
	\item The \emph{pointwise $2$-norm topology} on $\mathcal{B}_*(M)$ is induced by the seminorms $\lVert T \rVert^{2,\tau}_a =  \tau(T(a^*a))^{1/2}$ for all $a \in M$.
\end{itemize}
\end{dfn}

\begin{lem} \label{lem:top_on_Aut}
Let $M$ be a II$_1$ factor and denote the faithful, normal trace by $\tau$. The three topologies from Definition~\ref{dfn:top_on_Aut} agree on $\Aut{M}$. 
\end{lem}

\begin{proof} \label{pf:top_on_Aut}
It was proven in \cite[Corollary 3.8]{Ha75} that the $p$- and the $u$-topology coincide on $\Aut{M}$. By \cite[Proposition III.2.2.17]{Bl06} the $2$-norm topology on $M$ agrees with the strong topology on bounded subsets of $M$. Since an automorphism maps bounded subsets of $M$ to bounded subsets, the pointwise $2$-norm topology agrees with the pointwise strong topology on $\Aut{M}$ generated by the seminorms $\alpha \mapsto \lVert \alpha(x)\,\xi \rVert$ for all $x \in M$ and $\xi \in L^2(M,\tau)$. Since $\Aut{M}$ maps the unitary group $\mathcal{U}(M)$ into itself and the latter spans $M$, it suffices to consider $x \in \mathcal{U}(M)$. But this implies that the pointwise strong topology agrees with the pointwise $\sigma$-strong$^*$-topology, which in turn agrees with the $p$-topology on $\Aut{M}$ as stated in \cite[Section 1.4]{Wi98}.
\end{proof}

\begin{lem} \label{cor:transition_maps}
Let $U$ be a topological space and let $M$ be a II$_1$ factor. Consider $M$ to be equipped with the $2$-norm topology. Then there is a bijection between the continuous maps $\varphi \colon U \times M \to M$, such that $a \mapsto \varphi(x, a)$ is an automorphism for all $x \in U$ and the continuous maps $\hat{\varphi} \colon U \to \Aut{M}$, where $\Aut{M}$ is equipped with the $u$-topology. It is defined by $\hat{\varphi}(x) = \varphi(x, \,\cdot\,)$
\end{lem}

\begin{proof} \label{pf:transition_maps}
It is clear that the construction yields a bijection between the underlying sets; the only issue to check is continuity. By Lemma \ref{lem:top_on_Aut} the $u$-topology agrees with the pointwise $2$-norm topology. Suppose first that $\hat{\varphi}$ is continuous, i.e.\ $\hat{\varphi}(x_n)$ converges to $\hat{\varphi}(x)$ pointwise in $2$-norm for every net $(x_n)$ in $U$ that converges to $x \in U$. Let $(a_m)$ be a net in $M$ converging to $a \in M$ in $2$-norm. We have
\begin{align*}
	\lVert \varphi(x_n, a_m) - \varphi(x,a) \rVert_2 & \leq  \lVert \varphi(x_n, a_m - a) \rVert_2 + \lVert \hat{\varphi}(x_n)(a) - \hat{\varphi}(x)(a) \rVert_2 \\
	& \leq  \lVert a_m - a \rVert_2 + \lVert \hat{\varphi}(x_n)(a) - \hat{\varphi}(x)(a) \rVert_2
\end{align*}
where we used that an automorphism preserves the trace and is therefore isometric for the $2$-norm. This proves that $\varphi$ is continuous. Now suppose that $\varphi$ is continuous, then we have that $\lVert \hat{\varphi}(x_n)(a) - \hat{\varphi}(x)(a) \rVert_2 = \lVert \varphi(x_n,a) - \varphi(x,a) \rVert_2$ converges to zero for all $a \in M$.
\end{proof}

We will now construct the principal $\Aut{M}$-bundle $P_B \to X$ associated to a locally trivial topological bundle $(B,p)$ over the Hausdorff space $X$. Since we do not assume that the reader is familiar with the notion of principal $G$-bundles for a topological group $G$, we highlight the main points below. A good reference for this material is \cite[Chapter 4, Sections 2 and 3]{Hus94}.

\begin{dfn}
	Let $X$ be a topological space and let $G$ be a topological group. A (right) $G$-space $P$ together with a continuous $G$-map $q \colon P \to X$ (where $G$ acts trivially on $X$) is called a \emph{principal $G$-bundle}, if every point $x \in X$ has a neighbourhood $U \ni x$, such that there exists a $G$-equivariant homeomorphism $\phi_U \colon q^{-1}(U) \to U \times G$ with ${\rm pr}_U \circ \phi_U = \left.q\right|_{q^{-1}(U)}$. 
\end{dfn}

Let $(B,p)$ be a locally trivial topological bundle of tracial von Neumann algebras with fibre $M$. For any $x \in X$ there is an open neighbourhood $U \ni x$ and  homeomorphisms $\varphi$ and $\psi$ such that the diagram
\begin{equation}
\xymatrix
{
	p^{-1}(U) \ar[r]^{\varphi} \ar[d]_{p} & U \times M \ar[d]_{\pi_1} \\
	U \ar[r]^{\psi} & U \\
}	
\end{equation}
commutes, where $\pi_1: U \times M \rightarrow U$ is the projection onto the first coordinate. By replacing $\varphi$ with $(\psi^{-1} \times \mathrm{id}_M) \circ \varphi$, we get a commuting diagram of the form
\begin{equation}\label{TrivialisingNhoods}
\xymatrix
{
	p^{-1}(U) \ar[r]^{\cong} \ar[d]_{p} & U \times M \ar[d]_{\pi_1} \\
	U \ar[r]^{\mathrm{id}_U} & U. \\
}	
\end{equation}
We call such a $U$ a trivialising neighbourhood for $(B,p)$.

Consider $\Aut{M}$ as a topological group equipped with the $u$-topology. The principal $\Aut{M}$-bundle $P_B$ is obtained by replacing the fibre $M$ of $B$ by the group $\Aut{M}$ while preserving the transition maps. Write $B_x = p^{-1}(x)$ for the fibre at $x$ and $\Iso{M_1}{M_2}$ for the set of isomorphisms between two von Neumann algebras. As a set we define
\[
	P_B = \coprod_{x \in X} \Iso{M}{B_x}.
\]
Denote the canonical quotient map $P_B \to X$ by $q$. A local trivialisation $\varphi_U \colon U \times M \to p^{-1}(U)$ induces a bijection 
\[
	\psi_{U} \colon U \times \Aut{M} \to q^{-1}(U) = \left.P_B\right|_U = \coprod_{x \in U}\Iso{M}{B_x}.
\] 
Let $V \subseteq X$ be another subset with $U \cap V \neq \emptyset$ and such that there is a local trivialisation $\varphi_V \colon V \times M \to p^{-1}(V)$. Note that 
\[
	\left.\varphi_V^{-1} \circ \varphi_U\right|_{(U \cap V) \times M} \colon (U \cap V) \times M \to (U \cap V) \times M
\]
is of the form $(x,a) \mapsto (x,\varphi_{UV}(x, a))$ for a continuous map $\varphi_{UV} \colon (U \cap V) \times M \to M$ and $\varphi_{VU}^{-1}(x,a) = \varphi_{UV}(x,a)$. We have 
\[
	\left.\psi_V^{-1} \circ \psi_U\right|_{(U \cap V) \times \Aut{M}}(x, \alpha) = (x, \hat{\varphi}_{UV}(x) \circ \alpha).
\]
By Lemma \ref{cor:transition_maps} and the continuity of composition these maps are homeomorphisms. 

Now equip $P_B$ with the following topology: Cover $X$ by trivialising neighbourhoods $(U_i)_{i \in I}$ for $B$. A set $V \subseteq P_B$ is open if and only if for every point $y \in V$ there exists an $i \in I$ and a subset $V' \subseteq V \cap q^{-1}(U_i)$ such that $y \in V'$ and  $\psi_{U_i}^{-1}(V') \subseteq U_i \times \Aut{M}$ is an open neighbourhood of $\psi_{U_i}^{-1}(y)$. Since the transition maps $\psi_{U_j}^{-1} \circ \psi_{U_i} \colon (U_i \cap U_j) \times \Aut{M} \to (U_i \cap U_j) \times \Aut{M}$ are homeomorphisms, this definition is consistent. With this topology all maps $\psi_{U_i} \colon U_i \times \Aut{M} \to q^{-1}(U_i)$ become homeomorphisms. It is straightforward to check that this topology does not depend on the choice of trivialising cover and that $q \colon P_B \to X$ is a principal $\Aut{M}$-bundle.

Let $q \colon P \to X$ be a principal $\Aut{M}$-bundle and let $B = (P \times M)/\!\sim$ be the quotient with respect to the equivalence relation $(p \cdot \alpha,a) \sim (p, \alpha(a))$ for $\alpha \in \Aut{M}$. Let $\pi \colon B \to X$ given by $\pi([p,\alpha]) = q(p)$ be the quotient map. Observe that each point $([p_1,a_1], [p_2,a_2]) \in B \times_{\pi} B$ is equivalent to one of the form $([p_1,a_1], [p_1,a_2'])$ for a $a_2' \in M$ fixed by $p_1$, which allows us to define an addition and a multiplication via $[p_1, a_1 + a_2']$ and $[p_1, a_1 \cdot a_2']$ respectively. Since automorphisms preserve the norm and the trace, we obtain well-defined maps $\lVert\,\cdot\,\rVert \colon B \to [0,\infty)$ and $\tau \colon B \to \C$. 

Each point $x \in X$ has a neighbourhood $U$, such that the restriction $\left.P\right|_U$ is isomorphic to the trivial bundle $U \times \Aut{M}$. It follows that there is a homeomorphism $\left.B\right|_U \to U \times M$ compatible with the projection maps to $U$. Since the continuity conditions (i)--(viii) from Definition \ref{dfn:TopologicalBundle} can all be checked locally and are true for the trivial bundle, they hold for $B$ as well. It follows that $B$ is in fact a topological bundle of tracial von Neumann algebras. It is called the associated topological bundle. 

We shall show that these two constructions are inverse to one another. We need the following well-known fact about principal bundles: 
\begin{lem} \label{lem: sections}
	Let $X$ be a topological space and let $G$ be a topological group. Let $q \colon P \to X$ be a principal $G$-bundle. Suppose there exists a continuous section $\sigma \colon X \rightarrow P$. Then $P$ is isomorphic to the trivial principal $G$-bundle $X \times G$.
\end{lem}

\begin{proof}
	The trivialisation of $P$ is given by $\psi \colon X \times G \to P$ with $\psi(x,g) = \sigma(x)g$, which is clearly $G$-equivariant. To construct an inverse, let $P \times_q P = \{ (p_1, p_2) \in P \times P\ |\ q(p_1) = q(p_2) \} \subseteq P \times P$ and note that the map 
\[
	\kappa \colon P \times_q P \to G \quad ; \quad (p_1,p_2) \mapsto g_{12} \text{ with } p_1g_{12} = p_2 
\]
is well-defined and continuous, which can be checked using the local triviality of $P$. The inverse of $\psi$ is defined by 
\[
	\phi \colon P \to X \times G \quad ; \quad p \mapsto (q(p), \kappa(\sigma(q(p)), p)). \qedhere 
\]
\end{proof}

\begin{rem} \label{rem:hom_is_iso}
	In a similar fashion one can show that any continuous $G$-equivariant map $\varphi \colon P \to P'$ between principal bundles $q \colon P \to X$ and $q' \colon P' \to X$ such that $q' \circ \varphi = q$ is in fact an isomorphism. Such a map is said to cover the identity on $X$.
\end{rem}
\begin{prop} \label{lem:principal_vs_B}
Let $M$ be a II$_1$ factor and let $X$ be a Hausdorff space. The associated bundle construction yields a bijection between isomorphism classes of locally trivial topological bundles of tracial von Neumann algebras with fibre $M$ over $X$ and isomorphism classes of principal $\Aut{M}$-bundles over $X$.
\end{prop}

\begin{proof} 
Let $(B,p)$ be a locally trivial topological bundle of tracial von Neumann algebras and denote by $P_B$ the corresponding principal $\Aut{M}$-bundle. We need to check that the topological bundle associated to $P_B$ agrees with $B$. Consider the map 
\(
	(P_B \times M) / \!\sim\ \to B
\)
given by $[r,a] \mapsto r(a)$, where $r \in \Iso{M}{B_{q(r)}}$ and $a \in M$. To see that this is a homeomorphism, it suffices to check that it is a bijective local homeomorphism. It is straightforward to see that it is bijective. Any choice of local trivialisation of $B$, over $U \subseteq X$ say, induces a corresponding trivialisation of $P_B$ and we have 
\[
	\xymatrix{
		\left.(P_B \times M)/\!\sim\right|_U \ar[d]_-{\cong} \ar[r] & \left.B\right|_U \ar[d]^-{\cong} \\
		U \times (\Aut{M} \times M)/\!\sim \ar[r]_-{\cong} & U \times M
	}
\]
where the inverse of the lower horizontal map is given by $(x,a) \mapsto (x,[{\rm id}_M,a])$. Observe that both horizontal maps restrict to isomorphisms of tracial von Neumann algebras in each fibre. Hence they are isomorphisms of topological bundles in the sense of Definition~\ref{dfn:TopologicalBundle}.

Let $P$ be a principal $\Aut{M}$-bundle. We have to check that the principal $\Aut{M}$-bundle $P_B$ obtained from $B = (P \times M)/\!\sim$ agrees with $P$. By Remark \ref{rem:hom_is_iso} it suffices to construct a continuous $\Aut{M}$-equivariant map $P \to P_B$ covering the identity on $X$. This is defined by sending $r \in P$ to the isomorphism $\Iso{M}{B_{q(r)}}$ that maps $a$ to $[r,a] \in B$. Continuity is again easy to check in local trivialisations.
\end{proof}
Algebraic topology and sheaf theory provide tools for classifying principal $G$-bundles (see for example \cite[Chapter 4, Section 12]{Hus94}). For our purpose, we need only the following theorem.

\begin{thm} \label{cor:contractible_group}
	Let $X$ be a paracompact Hausdorff space and let $G$ be a contractible topological group. Let $q \colon P \to X$ be a principal $G$-bundle. Then $P$ is trivialisable. 
\end{thm}

\begin{proof}
	The assumptions about $P$, $X$ and $G$ imply that $q \colon P \to X$ has a global section by \cite[Lemma 4]{Di63}. Now apply Lemma \ref{lem: sections}. 
\end{proof}

\begin{cor} \label{thm:local_global} 
	Let $\M$ be a locally trivial W$^*$-bundle with all fibres isomorphic to the II$_1$ factor $M$. Assume $\Aut{M}$ is contractible with respect to the $u$-topology. Then $\M$ is trivial.
\end{cor}
\begin{proof} 
By the results of Section 3 and Proposition \ref{lem:principal_vs_B}, it suffices to show the corresponding principal $\Aut{M}$-bundle is trivial. This follows from Theorem \ref{cor:contractible_group}, since by assumption $\Aut{M}$ is contractible.
\end{proof}
Corollary \ref{thm:local_global} together with Popa and Takesaki's result  that $\Aut{\RR}$ is contractible in the $u$-topology \cite[Theorem 4]{Po93} gives our main theorem.
\begin{thm} \label{thm:hyperfinite_trivial}
	A locally trivial W$^*$-bundle with all fibres isomorphic to the hyperfinite II$_1$ factor $\RR$ is trivial.
\end{thm} 

\section{Non-trivial, locally trivial bundles} \label{sec:non-trivial}
In this section we give examples of non-trivial, but still locally trivial W$^*$-bundles over the circle $S^1$. The construction is motivated by the following facts from the theory of principal bundles: For every topological group $H$ there exists a topological space $BH$, such that isomorphism classes of principal $H$-bundles over a paracompact topological space $X$ are in bijection with homotopy classes of maps $X \to BH$ (see \cite[Theorem 3.1]{Mil56} and \cite[Proposition 6]{Hus94}). In particular, Proposition~\ref{lem:principal_vs_B} implies that isomorphism classes of locally trivial topological bundles of tracial von Neumann algebras with fibre a II$_1$ factor $M$ over the circle $S^1$ are in bijection with $[S^1,B\Aut{M}]$. The set $[S^1, B\Aut{M}]$ is in bijection with the set of conjugacy classes in the group\footnote{Note that it is the group structure of $\Aut{M}$ which turns $\pi_0(\Aut{M})$ into a group.} of path-components $\pi_0(\Aut{M})$. Since $\pi_0(\Aut{M})$ surjects onto $\pi_0(\Out{M})$, it therefore suffices to find factors, for which $\Out{M}$ is not path-connected to obtain non-trivial examples. Factors $M$ of type II$_1$ with $\Out{M}$ isomorphic to a prescribed compact group have been constructed by Ioana, Peterson and Popa in \cite{IPP08} in the abelian case and by Vaes and Falgui\`{e}res in \cite{VF08} for general compact groups. 

We will use the construction from \cite{VF08}: Fix a non-trivial finite group $G$. As sketched at the end of \cite[Section 2]{VF08}, there exists a minimal action of $G$ on $\RR$. By \cite[Corollary 2.2]{VF08} the group $\Gamma = SL(3,\Z)$ acts on the fixed point algebra $\RR^G$. Let $M = (\RR^G \rtimes \Gamma) \ast_{\RR^G} R$. The natural map $G \to \Aut{M}$ induces an isomorphism $G \cong \Out{M} = \Aut{M}/\text{Inn}(M)$ by \cite[Corollary 2.2]{VF08}. 



The group $\Out{M}$ is defined as a quotient and could in principle be non-Hausdorff. However, since $M$ is full, $\Out{M}$ is Hausdorff. Therefore the continuous bijection $G \to \Out{M}$ induced by the action is a homeomorphism. Let $\theta \colon \Aut{M} \to \Out{M} \cong G$ be induced by the quotient map and the above identification. Fix $g \in G$ and let $\alpha \in \Aut{M}$ be an automorphism with $\theta(\alpha) = g$. This choice induces a group homomorphism $\Z \to \Aut{M}$, which will also be denoted by $\alpha$. Let
\begin{equation} \label{eqn:B}
	B =  \R \times_{\alpha} M
\end{equation}
that is, take the product $\R \times M$ modulo the equivalence relation $(t + n, m) \sim (t, \alpha(n)(m))$ for all $n \in \Z$. Together with the canonical quotient map $B \to S^1$, this is a topological bundle of tracial von Neumann algebras over $S^1$ in the sense of Definition \ref{dfn:TopologicalBundle} and is locally trivial in the sense of Definition \ref{dfn:TopBundleLocTrivial}. We can, therefore, via Theorem \ref{TheOtherOne}, define a locally trivial W$^*$-bundle $\M$ which induces $B$.

\begin{lem} 
Let $G$ be a finite group, let $M$ be the II$_1$ factor with $\Out{M} \cong G$ constructed above and let $\theta \colon \Aut{M} \to \Out{M} \cong G$ be the quotient map. Let $\alpha \in \Aut{M}$ with $g = \theta(\alpha) \neq e$. Then the W$^*$-bundle $\M$ associated to the topological bundle $B$ given by (\ref{eqn:B}) is non-trivial.
\end{lem}

\begin{proof}
Let $q \colon P \to S^1$ be the principal $\Aut{M}$-bundle of $B$. Suppose for the sake of contradiction that $\M$ is trivialisable. By the results of Section~3, $P$ is trivialisable. Consider $Q = P \times_{\theta} G$ defined as the quotient of the product $P \times G$ with respect to the equivalence relation $(p \cdot\beta, h) \sim (p, \theta(\beta)\cdot h)$ for $\beta \in \Aut{M}$. If $P$ is trivialisable, so is $Q$, but $Q \to S^1$ is a principal $G$-bundle over $S^1$ for the finite group $G$. By elementary covering space theory the isomorphism classes of these are in correspondence with the conjugacy classes of $G$. 

More precisely, the conjugacy class associated to $Q$ can be obtained as follows: Fix a basepoint $q_0 \in Q$ and lift the quotient map $[0,1] \to S^1$ to a continuous path $\gamma \colon [0,1] \to Q$ with $\gamma(0) = q_0$. By the path lifting property such a lift exists and is unique. Since $\gamma(0)$ and $\gamma(1)$ lie in the same fibre, there is a unique $h \in G$, such that $\gamma(0) = \gamma(1)\cdot h$. The conjugacy class of $h \in G$ is independent of $q_0$. 

The bundle $Q$ constructed above corresponds to the class of $g \in G$, whereas the trivial bundle corresponds to the conjugacy class of the neutral element $e \in G$, which only contains $e$, in contradiction with $g \neq e$. Therefore $\M$ can not be trivial.
%
\end{proof}

\begin{rem}
The above construction can easily be extended to non-trivial locally trivial W$^*$-bundles over more general spaces than $S^1$. Let $G$ be a finite group, let $X$ be a compact Hausdorff space, such that there exists a non-trivial principal $G$-bundle $\overline{X} \to X$. Since $G$ is discrete, $\overline{X}$ is just a covering space of $X$. Let $M$ be the II$_1$ factor constructed above. In particular, we have an isomorphism $\varphi \colon G \to \Out{M}$. There is an obstruction to lifting $\varphi$ to a homomorphism $\hat{\varphi} \colon G \to \Aut{M}$, which lives in $H^3(G, \mathcal{U}(1))$. By \cite[Theorem 4.1.3]{Su80} it is the only obstruction. Suppose it vanishes, then the topological bundle
\(
	B = \overline{X} \times_{\hat{\varphi}} M
\) 
is non-trivial. In fact, let $P$ be the associated principal $\Aut{M}$-bundle and denote by $\theta \colon \Aut{M} \to G$ the quotient map. By construction we have an isomorphism 
\(
	\overline{X} \cong P \times_{\theta} G
\) 
of principal $G$-bundles. Since $\overline{X}$ was supposed to be non-trivial, the same holds true for $P$. Thus, the associated W$^*$-bundle is also non-trivial.
\end{rem}

\end{document}